\documentclass[11pt]{amsart}
\usepackage{amsmath,amssymb,amsthm,graphicx, url, verbatim, float,hyperref,ulem} 
\usepackage{enumerate}
\usepackage{svg}
\usepackage{xcolor, relsize}
\usepackage{comment}

\newtheorem{lemma}{Lemma}[section]
\newtheorem{proposition}[lemma]{Proposition}
\newtheorem{theorem}[lemma]{Theorem}

\newtheorem{corollary}[lemma]{Corollary}
\newtheorem{question}[lemma]{Question}
\newtheorem{conjecture}[lemma]{Conjecture}

\newcommand{\bcon}{\begin{conjecture}}
\newcommand{\econ}{\end{conjecture}}
\newcommand{\bcor}{\begin{corollary}}
\newcommand{\ecor}{\end{corollary}}
\newcommand{\bdf}{\begin{definition}}
\newcommand{\edf}{\end{definition}}
\newcommand{\benu}{\begin{enumerate}}
\newcommand{\eenu}{\end{enumerate}}
\newcommand{\beq}{\begin{equation}}
\newcommand{\eeq}{\end{equation}}
\newcommand{\bexa}{\begin{example}}
\newcommand{\eexa}{\end{example}}
\newcommand{\bexe}{\begin{exercise}}
\newcommand{\eexe}{\end{exercise}}
\newcommand{\bfac}{\begin{fact}}
\newcommand{\efac}{\end{fact}}
\newcommand{\bite}{\begin{itemize}}
\newcommand{\eite}{\end{itemize}}
\newcommand{\blem}{\begin{lemma}}
\newcommand{\elem}{\end{lemma}}
\newcommand{\bmat}{\begin{matrix}}
\newcommand{\emat}{\end{matrix}}
\newcommand{\bprb}{\begin{problem}}
\newcommand{\eprb}{\end{problem}}
\newcommand{\bpro}{\begin{proposition}}
\newcommand{\epro}{\end{proposition}}

\newcommand{\bque}{\begin{question}}
\newcommand{\eque}{\end{question}}
\newcommand{\brem}{\begin{remark}}
\newcommand{\erem}{\end{remark}}
\newcommand{\bthm}{\begin{theorem}}
\newcommand{\ethm}{\end{theorem}}
\newtheorem*{namedtheorem}{\theoremname}

\newcommand{\bpr}{\begin{proof}}
\newcommand{\epr}{\end{proof}}

\theoremstyle{definition}
\newtheorem{definition}[lemma]{Definition}
\newtheorem{remark}[lemma]{Remark}
\newtheorem{example}[lemma]{Example}

\newcommand{\theoremname}{testing}

\newcommand{\Z}{\mathbb{Z}}

\newcommand{\Q}{\mathbb{Q}}
\newcommand{\C}{\mathbb{C}}

\newcommand{\slC}{\mathrm{SL}_2(\mathbb{C})}

\title[Non-semisimple invariants and abelian classical shadows]{Non-semisimple quantum invariants and abelian classical shadows}
\author{Renaud Detcherry}

\address{Institut de Mathématiques de Bourgogne, UMR 5584 CNRS, Université Bourgogne Franche-Comté, F-2100 Dijon, France}
\email{renaud.detcherry@u-bourgogne.fr}

\begin{document}
\begin{abstract}Using the $U_q^Hsl_2$ non-semisimple invariants of $3$-manifolds at odd roots of unity, we construct maps on the Kauffman bracket skein module at $\zeta$ a root of unity of order twice an odd number, having any possible abelian non-central character as classical shadow.
\end{abstract}
\maketitle	
\section{Introduction}
\label{sec:intro}

For $\zeta$ a root of unity of order $2N$ where $N>3$ is odd, and $M$ a compact oriented $3$-manifold, let $S_{\zeta}(M)$ be the Kauffman bracket skein module evaluated at $A=\zeta.$ Let also $X(M)=\mathrm{Hom}(\pi_1(M),\slC  )//\slC$ be the $\slC$-character variety of $M.$ Skein modules of $3$-manifolds at roots of unity are deeply connected to $X(M),$ as discovered by Bonahon and Wong \cite{BW16}. Inspired by \cite{BW16}, we define:
\begin{definition}
 For $\chi \in X(M),$ we will say that a surjective map $f:S_{\zeta}(M) \rightarrow \C$ has \textit{classical shadow} $\chi$ if for all link $L$ and knot $K,$ we have 
$$f(L\cup T_N(K))=-\chi(K) f(L),$$
where $T_N$ is the $N$-th Chebychev polynomial, which is the unique polynomial in $\Z[z]$ satisfying $T_N(X+X^{-1})=X^N+X^{-N}.$
\end{definition}

The purpose of this note is to show the following:

\begin{theorem}
	\label{thm:main} For any $3$-manifold $M$ and any abelian non-central central $\chi \in X(M)$ there exists a surjective map $f:S_{\zeta}(M) \longrightarrow \C$ with classical shadow $\chi.$
\end{theorem}
While Theorem \ref{thm:main} may be deduced from the results in the recent preprint \cite{KK22} about the Azumuya locus of skein algebras, we will construct the map $f$ in an alternative way. Indeed, we will show that such maps can be recovered from a version of Costantino--Geer--Patureau-Mirand's non semisimple quantum invariants \cite{CGP14}.

Building maps on $S_{\zeta}(M)$ realizing each possible classical shadow was a key ingredient in previous work of the author and Kalfagianni and Sikora \cite{DKS}, to give formulas for the dimension of skein modules of $3$-manifolds over $\Q(A).$ For central representations, such maps could be built from the (semisimple) Reshetikhin-Turaev invariants of $3$-manifolds, hence the present paper can be seen as the analogous construction in the non-semisimple case. 

The paper is organized as follows: Section \ref{sec:modules} and \ref{sec:trivgraphs} discuss the category of weight modules over the unrolled quantum group of $sl_2$ at odd roots of unity, and Section \ref{sec:3mfd-inv} recalls the construction  of \cite{CGP14} of non-semisimple $3$-manifolds invariants. Much of the material in those sections is directly inspired from \cite{CGP14} and \cite{CGP15}, who treated the case of even roots of unity. We don't claim any originality in those sections. Finally, Section \ref{sec:maps} is the heart of the paper and proves Theorem \ref{thm:main}. 

\textbf{Acknowledgements:} The author thanks François Costantino for helpful discussions, and Effie Kalfagianni and Adam Sikora for helpful discussions and comments on a preliminary version of this note. Over the course of this work, the first author was supported by the projects "NAQI-34T" (ANR-23-ERCS-0008) and ”CLICQ” of the R\'egion Bourgogne Franche-Comt\'e.

\section{Modules over $U_q^H\mathrm{sl}_2$}
\label{sec:modules}

Let $N$ be an odd integer and let $q=e^{\frac{4i\pi}{N}}.$ Note that $q$ is a primitive $N$-th root of unity, and that $\zeta=-e^{\frac{2i\pi}{N}}$ is a primitive $2N$-th root of unity. 

For $z\in \C$ we will also write $q^z$ for $e^{\frac{4iz\pi}{N}},$ and set:
$$\lbrace z \rbrace:=q^z-q^{-z}, \ [z]:=\frac{q^z-q^{-z}}{q-q^{-1}}.$$

We note that $[z]=0$ if and only if $z\in \frac{N}{4}\Z.$

We recall that the unrolled quantum group $U_q^Hsl_2$ is the algebra generated by $E,F,K,H$ with relations:
$$KE=q^2 EK, \ KF=q^{-2}FK, \ [H,E]=2E, \ [H,F]=-2F, \ [H,K]=0, \ [E,F]=\frac{K-K^{-1}}{q-q^{-1}}.$$
It has an Hopf algebra structure, with coproduct 
$$\Delta(E)=E\otimes 1 +K\otimes E, \ \Delta(F)=F\otimes 1 + K^{-1} \otimes F, \ \Delta(K)=K\otimes K, \Delta(H)=H\otimes 1 + 1 \otimes H,$$
and  counit $\varepsilon:$ 
$$\varepsilon(E)=\varepsilon(F)=\varepsilon(H)=0, \ \varepsilon(K)=1,$$
and antipode:
$$S(E)=-EK^{-1}, \ S(F)=-KF, \ S(K)=K^{-1}, S(H)=-H.$$

A weight module $V$ for $U_q^H\mathrm{sl}_2$ will be a finite dimensional $U_q^H\mathrm{sl}_2$-module which is spanned by eigenvectors of $H,$ and on which $K$ acts as $q^H.$ An eigenvector of $H$ with eigenvalue $\lambda$ will be called a weight vector of weight $\lambda.$

Let $\mathcal{C}$ be the monoidal category of $U_q^H\mathrm{sl}_2$-weight modules. For $\overline{\alpha} \in \C/\Z,$ let $\mathcal{C}_{\overline{\alpha}}$ be the full subcategory of weight modules spanned by weight vectors with weights $\beta$ such that $\beta \equiv \alpha \ (\mathrm{mod} \ \Z).$

Since $\Delta(H)=H\otimes 1 + 1 \otimes H,$ we have $\mathcal{C}_{\overline{\alpha}}\otimes \mathcal{C}_{\overline{\alpha'}} \subset \mathcal{C}_{\overline{\alpha + \alpha'}},$ for any $\overline{\alpha},\overline{\alpha'}\in \C/\Z.$

Also, since $S(H)=-H,$ the dual of an object in $\mathcal{C}_{\overline{\alpha}}$ belongs to $\mathcal{C}_{\overline{-\alpha}}.$ Therefore $\mathcal{C}$ a $\C/\Z$-graded monoidal category.

A braiding $c_{V,W}$ between two objects $V,W\in \mathcal{C}$ can be defined (see \cite{GPM13}) by the formula $c_{V,W}=\tau \circ R,$ where $\tau(v\otimes w)=w \otimes v$ and $R: V\otimes W \longrightarrow V\otimes W$ is the map defined by the formula:
\begin{equation}\label{eq:Rmatrix}R=q^{(H\otimes H)/2}\underset{n=0}{\overset{\infty}{\sum}}q^{n(n-1)/2}\frac{(q-q^{-1})^n}{[n]!}E^n\otimes F^n,\end{equation}
where $q^{(H\otimes H)/2}$ acts as multiplication by $q^{\lambda\lambda'/2}$ on $v\otimes w$ when $v,w$ are weight vectors of weights $\lambda,\lambda'.$

Note that despite the infinite sum, this yields a well defined linear map on $V\otimes W.$  Indeed, the actions of $E$ and $F$ on any weight module is nilpotent, since $E$ (resp. $F$) sends a weight vector of weight $\lambda$ to a weight vector of weight $\lambda+2$ (resp. $\lambda-2$).

Moreover, a twist $\theta_V$ is defined (see \cite{GPM13}) on any object of $V$ by the formula:
\begin{equation}\theta_V^{-1}=K^{N-1} \underset{n=0}{\overset{N-1}{\sum}}\frac{(q-q^{-1})^n}{[n]!}q^{n(n-1)/2}S(F^n)q^{-H^2/2}E^n,\label{eq:twist}\end{equation}
where again $q^{-H^2/2}$ acts on a weight vector $v$ of weight $\lambda$ as multiplication by $q^{-\lambda^2/2}.$

Finally, the category $\mathcal{C}$ is \textit{pivotal}, for the following evaluations and coevaluations morphisms: if $\lbrace v_i \rbrace$ is a basis of $V\in \mathcal{C}$ and $\lbrace v_i^* \rbrace$ is the dual basis of $V^*,$ we set

$$coev_V :\C \longrightarrow V\otimes V^* , \ \textrm{where} \ coev_V(1)=\underset{i}{\sum} v_i \otimes v_i^*$$
$$ev_V: V^*\otimes V \longrightarrow \C, \ \textrm{where} \ ev_V(f\otimes v)=f(v)$$
and $\widetilde{coev}_V=(Id_{V^*}\otimes \theta_V)\circ c_{V,V^*} \circ coev_V,$ $\widetilde{ev}_V=ev_V \circ c_{V,V^*}\circ (\theta_V\otimes Id_{V^*}).$
Plugging in Equation \ref{eq:Rmatrix} and \ref{eq:twist}, one gets
$$\widetilde{coev}_V(1)=\underset{i}{\sum}K^{N-1}v_i \otimes v_i^*$$
$$\widetilde{ev}_V(v\otimes f)=f(K^{1-N}v_i).$$

The \textit{quantum trace} of a morphism $f\in \mathrm{End}(V)$ is then by definition 
$$qtr(f)=\widetilde{ev}_V \circ (Id_{V^*}\otimes f)\circ coev_V(1)=\mathrm{Tr}(K^{1-N}f).$$
and the \textit{quantum dimension} of an object $V\in \mathcal{C}$ is $qtr(Id_V).$

We will need to define some specific $U_q^Hsl_2$-modules:
\begin{definition}\label{def:modules} We define modules $V_{\alpha},S_n,P_n,\varepsilon \in \mathcal{C}$ by:
\begin{enumerate}
	\item For $\alpha\in \C,$ the module $V_{\alpha} \in \mathcal{C}_{\overline{\alpha}}$ has basis $\lbrace v_0,\ldots,v_{N-1}\rbrace$ and module structure:
	$$Hv_j=(\alpha+N-1-2j)v_j, \ Fv_j=v_{j+1}, \ Fv_{N-1}=0, Ev_j=[j][\alpha-j]v_{j-1}, \ Ev_0=0.$$
	
	\item For $n\in \lbrace 0,\ldots, N-1\rbrace,$ the module $S_n\in \mathcal{C}_{\overline{0}}$ has basis $\lbrace e_0,\ldots,e_n\rbrace$ and module structure:
	$$He_j=(n-2j)e_j, \ Fe_j=e_{j+1}, \ Fe_n=0,\ Ee_j=[j][n+1-j]e_{j-1}, \ Ee_0=0.$$
	
	\item For $n\in \lbrace 0 ,\ldots,N-2\rbrace,$ the module $P_n\in \mathcal{C}_{\overline{0}}$ has basis $\lbrace x_0,\ldots,x_{N-1},y_0,\ldots,y_{N-1}\rbrace$ and module structure:
	$$Hx_j=(2N-2-n-2j)x_j, \ Fx_j=x_{j+1}, \ Fx_{N-1}=0, \ Ex_j=-[j][j+1+n]x_{j-1}, \ Ex_0=0$$
	$$Hy_j=(n-2j)y_j, \ Fy_j=y_{j+1}, \ Fy_{N-1}=0, \ Ey_0=x_{N-2-n}, \ Ey_{n+1}=x_{N-1},$$
	$$Ey_j=[j][n+1-j]y_{j-1} \ \textrm{for} \ j \neq 0,n+1.$$ 
	\item For $k\in \Z,$ the module $\varepsilon^l$ is a one dimensional module spanned by $v$ such that $Ev=Fv=0,$ and $Hv=l\frac{N}{4}v$ (and thus $Kv=(-1)^lv$). Note that $\varepsilon^l \in \mathcal{C}_{\overline{lN/4}}.$
\end{enumerate}
\end{definition}
Let us introduce the notation 
$$\overset{..}{\C}=\left(\C \setminus \frac{1}{4}\Z \right)\cup \frac{N}{4}\Z.$$
The classification of $U_q^H\mathrm{sl}_2$-weight modules is given as follows:

\begin{theorem}\label{thm:classifModules}\cite{CGP15}
	The indecomposable modules of $\mathcal{C}$ are the modules $V_{\alpha}$ for $\alpha \in \overset{..}{\C},$ the module $S_n \otimes \varepsilon^l$ for $n\in \lbrace 0,\ldots N-2 \rbrace$ and $l\in \Z,$ and the modules $P_n\otimes \varepsilon^l$ for $n\in \lbrace 0,\ldots N-2 \rbrace$ and $l\in \Z.$ 
	
	Moreover, $\varepsilon$ is an invertible module, $S_n$ and $V_{\alpha}$ are simple modules, $V_{\alpha}$ and $P_n$ are projective modules.
	
\end{theorem}

\textbf{Sketch of proof:}
	Our conventions differ slightly from the ones used in \cite{CGP15}, by the fact that we used $q=e^{\frac{4i\pi}{N}}$ instead of $q=e^{\frac{2i\pi}{N}}$ and the grading group is $\C/\Z$ instead of $\C/2\Z.$ However, from the fact that $[\alpha]=0$ if and only if $\alpha \in \frac{N}{4}\Z,$ one can readily check that all the modules $S_n$ and the modules $V_{\alpha}$ for $\alpha \in \overset{..}{\C}$ are generated by any of their weight vectors and hence are simple. Then one can check that the family 
	$$\lbrace V_{\alpha}, \alpha \in \overset{..}{\C}\rbrace \cup \lbrace S_n \otimes \varepsilon^l, n\in \lbrace 0,\ldots,N-1\rbrace, l\in \Z \rbrace$$ provides a simple object of any possible highest weight, thus contains all simple objects in $\mathcal{C}.$
	Similarly, one can adapt the arguments in \cite{CGP15} to check that the only missing indecomposable modules are the $P_n \otimes \varepsilon^l$ and that they are projective.
$\square$

\medskip

We will also need the following lemma about tensor products of such modules:

\begin{lemma}\label{lemma:tensorprod}
	Let $I$ be the set of modules which are direct sums of modules in the family $$\lbrace V_{\alpha}, \alpha \in \overset{..}{\C} \rbrace \cup \lbrace P_i\otimes \varepsilon^l, 0\leq i \leq N-2, l \in \Z \rbrace.$$ Then $I$ is an ideal for the tensor product operation.
\end{lemma}

\begin{proof}From Theorem \ref{thm:classifModules}, $I$ corresponds exactly to the set of projective modules in $\mathcal{C}.$ Then this follows from the fact that $\mathcal{C}$ is a pivotal category and Lemma 17 of \cite{GPV}, which states that projective objects in a pivotal category form an ideal.
\end{proof}
The following definition will be useful:
\begin{definition}
	\label{def:character} For $V \in \mathcal{C}$ the character $\chi_V$ of $V$ is the following element of $\Z[\C]:$
	$$\chi_V(X)=\underset{t\in \C}{\sum} (\dim V(t))X^t,$$
	where $V(t)$ is the subspace of weight vectors of weight $t.$
\end{definition}
Note that from the fact that $\Delta(H)=H\otimes 1 +1\otimes H,$ it follows from the character is multiplicative: $\chi_{V\otimes W}=\chi_V\chi_W$ for any $V,W\in \mathcal{C}.$

\begin{lemma}
	\label{lemma:tensorprod2} We have, for any $\alpha,\beta \in \overset{..}{\C}$ such that $\alpha+\beta \notin \frac{1}{4}\Z,$
	$$V_{\alpha}\otimes V_{\beta}\simeq \underset{k\in H_N}{\bigoplus} V_{\alpha+\beta+k}.$$
	Moreover, for any $\alpha\in \overset{..}{\C},$ we have
	$$V_{\alpha}\otimes V_{-\alpha} \simeq V_0\otimes V_0 \simeq V_0 \oplus \underset{n=0}{\overset{\frac{N-3}{2}}{\bigoplus}}P_{2n}.$$
\end{lemma}
\begin{proof}
	Notice that $\chi_{V_{\alpha}}(X)=X^\alpha \left(\frac{X^N-X^{-N}}{X-X^{-1}}\right)$ and $\chi(P_n)(X)=(X^{N-1-n}+X^{n+1-N})\left(\frac{X^N-X^{-N}}{X-X^{-1}}\right).$
	It follows that the characters of projective indecomposable modules in $\mathcal{C}$ are linearly independent, and that a projective module is determined by its character.
	Hence the tensor product decompositions follow from the corresponding easy to check identities on the characters:
	$$\chi_{V_{\alpha}}\chi_{V_{\beta}}=\underset{k \in H_N}{\sum} \chi_{V_{\alpha+\beta+k}}$$
	$$\chi_{V_{\alpha}}\chi_{V_{-\alpha}}=\chi_{V_0}^2=\chi_{V_0}+\underset{n=0}{\overset{\frac{N-3}{2}}{\sum}} \chi_{P_{2n}}.$$
\end{proof}

Finally, we will need the description of the endomorphism spaces of indecomposable modules. 
When $W$ is simple, $\mathrm{End}(W)$ is just spanned by $id_W.$ So, let us also describe $\mathrm{End}(P_n),$ for $n\in \lbrace 0,\ldots,N-2 \rbrace.$

\begin{lemma}
	\label{lemma:endPn}
	For $n\in \lbrace 0,\ldots N-2\rbrace,$ let $h_n$ be the non-zero endomorphism of $P_n$ defined by $h_n(y_i)=x_{N-1-n+i}$ for $i\in \lbrace 0,\ldots,n \rbrace$ and $h_n$ vanishes on all other weight vectors of $P_n.$ Then $h_n$ is nilpotent of order $2,$ and moreover, we have 
	$$\mathrm{End}(P_n)=\C Id_{P_n} \oplus \C h_n.$$
\end{lemma}

\begin{proof}
	It is clear that $h_n$ is nilpotent of order $2,$ and the fact that $h_n\in \mathrm{End}(P_n)$ can be directly checked from the formulas of Definition \ref{def:modules} for the module structure on $P_n.$ Thus we need only to check that $\mathrm{End}(P_n)$ is spanned by $id_{P_n}$ and $h_n.$
	
	To see that, first we claim $P_n$ is generated by  $y_0.$ Indeed, if $y_0$ belongs to a submodule $V$ of $P_n,$ applying $F$ repeatedly we get that $V$ contains all weight vectors $y_i.$ Moreover, $V$ also contains $x_{N-2-n}=Ey_0.$ Applying $E$ repeatedly we get that $V$ contains $x_0$ (since $[i][n+1+i]\neq 0$ for $i\in \lbrace 1,\ldots,N-2-n\rbrace$), and then applying $F$ repeatedly we get that $V$ contains all weight vectors $x_i.$
	
	Now, an endomorphism $f$ of $P_n$ must map the weight vectors $y_0$ to a weight vector of the same weight, so one must have $f(y_0)=\lambda y_0+ \mu x_{N-1-n}$ for some constants $\lambda,\mu \in \C.$ Since $y_0$ generates $P_n,$ this implies that $f=\lambda Id_{P_n}+\mu h_n.$
\end{proof}
\section{Invariants of framed trivalent graphs}
\label{sec:trivgraphs}

The category $\mathcal{C}$ being a ribbon category, we get from the Reshetikhin-Turaev functor an invariant of framed oriented links with components colored by objects of $\mathcal{C},$ see \cite[Theorem 2.5]{Turaev:book}.

We will need to express the invariant of some open Hopf links, where the open Hopf link colored by $V,W$ is the following framed oriented tangle:

\begin{center}
	\includegraphics[width=2cm]{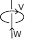}
\end{center}

For $W$ a projective module and $V\in \mathcal{C},$ let $\Phi_{V,W}$ be the element of $\mathrm{End}(W)$ corresponding to the Hopf link colored by $V,W$ opened at the edge colored by $W.$ If $W$ is simple, let also $\langle \Phi_{V,W} \rangle \in \C$ be defined by $\Phi_{V,W}=\langle \Phi_{V,W} \rangle  id_W.$

\begin{lemma}\label{lemma:Hopf}
	We have for each $\alpha,\beta \in \overset{..}{\C}:$
	$$\Phi_{S_1,V_{\alpha}}=(q^{\alpha}+q^{-\alpha})Id_{V_{\alpha}}, \ \Phi_{V_{\alpha},V_{\beta}}= \frac{q^{\alpha\beta}}{d(\beta)}Id_{V_{\beta}},$$
where $d(\beta)=\frac{\lbrace \beta \rbrace}{\lbrace N \beta \rbrace}.$	
	and for each $n\in \lbrace 0,\ldots N-2\rbrace:$
	$$\Phi_{S_1,P_n}=\left(q^{n+1}+q^{-n-1}\right)Id_{P_n} + (q-q^{-1})^2 h_n$$
	for some scalar $\lambda_{n}\in \C.$
\end{lemma}
\begin{proof}
	The proof of those formulas is an adaptation of \cite[Lemma 6.6]{CGP15}. The map $\Phi_{V,W}$ can be computed as the  partial quantum trace of $c_{V,W}\circ c_{W,V}$ relative to $V.$ (recall that the quantum trace of $f$ is the trace of $K^{1-N}f$) When $W$ is simple, one must have $\Phi_{V,W}=\lambda Id_W$ so it is sufficient to compute $\Phi_{V,W}$ on a highest weight vector $w$ of $W.$ However, since $Ew=0,$ one has that $c_{V,W}\circ c_{W,V}$ acts on any vector $w\otimes v$ where $v\in V$ as $q^{H\otimes H}.$ It follows that $\lambda=\chi_V(q^{\lambda+1-N}),$ where $\lambda$ is the highest weight of $W,$ and $\chi_V(X)$ is the character of $V,$ defined in Definition \ref{def:character}.
	
	For $V=S_1$ and $W=V_{\alpha},$ since $\alpha+N-1$ is the highest weight of $V_{\alpha}$ and $\chi_{S_1}(X)=X+X^{-1},$ one gets $\Phi_{S_1,V_{\alpha}}= (q^{\alpha}+q^{-\alpha})Id_{V_{\alpha}}.$
	
	For $V=V_{\alpha}$ and $W=V_{\beta},$ since $\beta+N-1$ is the highest weight of $V_{\beta}$ and ${\chi_{V_{\alpha}}(X)=X^{\alpha}\left(\frac{X^N-X^{-N}}{X-X^{-1}}\right),}$ one gets
	$$\Phi_{V_{\alpha},V_{\beta}}=q^{\alpha\beta}\frac{\lbrace N\beta\rbrace}{\lbrace \beta \rbrace}Id_{V_{\beta}}=\frac{q^{\alpha\beta}}{d(\beta)}Id_{V_{\beta}},$$

	Finally, we compute $\Phi_{S_1,P_n}.$ Recall that $S_1$ has basis $e_0,e_1,$ for the action given in Definition \ref{def:modules}. Let $f=c_{S_1,P_n}\circ c_{P_n,S_1}.$ 
	Since $\Phi_{S_1,P_n}$ is the partial quantum trace of $f$ relative to $S_1,$ if for $v\in P_1$ we have $f(v\otimes e_0)=v_1\otimes e_0 +v_2 \otimes e_1$ and $f(v\otimes e_1)=v_3\otimes e_0 +v_4\otimes e_1,$ then $\Phi_{S_1,P_i}(v)=qv_1+q^{-1}v_4.$ 
	
	Moreover, the coefficient of $Id_{P_n}$ (resp. $h_n$) in $\Phi_{S_1,P_i}$ is the same as the coefficient of $y_0$ (resp. $x_{N-1-n}$) in $\Phi_{S_1,P_i}(y_0).$ Let us compute $\Phi_{S_1,P_i}(y_0).$
	We have
	\begin{multline*}f(y_0\otimes e_0)=c_{S_1,P_n}\circ c_{P_n,S_1}(y_0\otimes e_0)=c_{S_1,P_n}(q^{n/2}e_0\otimes y_0 +q^{-(n+2)/2}(q-q^{-1})e_1\otimes x_{N-2-n})
	\\=\left(q^n y_0 + q^{-1}(q-q^{-1})^2 x_{N-1-n}\right)\otimes e_0 + \ldots \otimes e_1
	\end{multline*}
	and
	$$f(y_0\otimes e_1)=c_{S_1,P_n}\circ c_{P_n,S_1}(y_0 \otimes e_1)=c_{S_1,P_n}(q^{-n/2}e_1\otimes y_0)=q^{-n}y_0\otimes e_1 + \ldots \otimes e_0.$$
	Therefore, we have $\Phi_{S_1,P_i}(y_0)=(q^{n+1}+q^{-n-1})y_0 +(q-q^{-1})^2x_{N-1-n},$ which concludes the proof.
\end{proof}

We can formally extend the invariant of framed oriented tangles colored by elements of $\mathcal{C}$ to an invariant of framed oriented tangles colored by formal linear combination of elements in $\mathcal{C}.$ The invariant is then obtained by extending multilinearly. We define the sequence of Chebychev polynomials $T_n(z) \in \Z[z]$ by $T_0(z)=2, T_1(z)=z, T_{n+1}(z)=zT_n(z)-T_{n-1}(z).$ We recall that $T_n(z)$ satisfies the identity $T_n(X+X^{-1})=X^n+X^{-n},$ as can be shown by an easy induction.

\begin{lemma}\label{lemma:HopfChebychev}
	We have for each $\alpha \in \overset{..}{\C}:$
	$$\Phi_{T_N(S_1),V_{\alpha}}=\left(e^{4i\pi\alpha}+e^{-4i\pi\alpha}\right)Id_{V_{\alpha}}$$
	and for each $i\in \lbrace 0,\ldots N-2\rbrace:$
	$$\Phi_{T_N(S_1),P_i}=2 Id_{P_i}.$$
	
\end{lemma}
\begin{proof}
The first claim follows by Lemma \ref{lemma:Hopf} and the fact that $T_N(X+X^{-1})=X^N+X^{-N},$ for any invertible element $X$ in a ring.

For the second claim, first note that if $\lambda,t$ are elements in some ring such that $t^2=0,$ then for any polynomial $R\in \Z[X]$ we have that $R(\lambda +t)=R(\lambda)+R'(\lambda)t.$

Applying this to $\Phi_{T_N(S_1),P_i},$ we have that 
	$$\Phi_{T_N(S_1),P_i}=T_N(\Phi_{S_1,P_i})=T_N(q^{i+1}+q^{-i-1})Id_{P_i} +T_N'(q^{i+1}+q^{-i-1})(q-q^{-1})^2 h_i=2Id_{P_i}.$$
	Here we have used that $T_N(q^{i+1}+q^{-i-1})=q^{N(i+1)}+q^{-N(i+1)}=2$ and that ${T_N'(q^{i+1}+q^{-i-1})=0.}$ The second equality comes from the identity $T_N'(X+X^{-1})=N\left(\frac{X^N-X^{-N}}{X-X^{-1}}\right),$ 	which can be deduced from the identity $T_N(X+X^{-1})=X^N+X^{-N}$ by differentiating.
\end{proof}

For $V\in \mathcal{C},$ and $f\in \mathrm{End}(V\otimes V),$ let 
$$t_L(f)=\widetilde{ev}_V \otimes Id_V)\circ (Id_{V^*}\otimes f)\circ (coev_V \otimes Id_V) \in \mathrm{End}(V)$$
$$t_R(f)=(Id_V\otimes ev_V)\circ (f\otimes Id_{V^*})\circ (Id_V\otimes \widetilde{ev}_V)\in \mathrm{End}(V).$$
\begin{definition}\label{def:ambidex}\cite{GPT}
	A simple object $V\in \mathcal{C}$ is ambidextrous if 
	$$\forall f\in \mathrm{End}(V\otimes V), \ t_L(f)=t_R(f).$$	
\end{definition}
\begin{lemma}\label{lemma:ambidex}
	$V_0$ is an ambidextrous object in $\mathcal{C}.$
\end{lemma}
\begin{proof}
	By \cite[Lemma 1]{GPT}, an object $V$ in a ribbon category is ambidextrous if $c_{V,V}$ commutes with any element of $\mathrm{End}(V\otimes V).$ However, by Lemma \ref{lemma:tensorprod2} and Lemma \ref{lemma:endPn}, 
	$$\mathrm{End}(V_0\otimes V_0)\simeq \mathrm{End}(V_0) \oplus \underset{n=0}{\overset{\frac{N-3}{2}}{\bigoplus}} \mathrm{End}(P_{2n})$$ is a commutative ring.
\end{proof}

Given an ambdextrous object $J$, \cite{GPT} construct a \textit{ambidextrous pair} in the following way. An ambidextrous pair $(A,d)$ will consist of $A$ a set of simple objects $V$ such that $\Phi_{J,V} \neq 0,$ and $d:A \longrightarrow \C$ is defined by the formula: 

$$d(V)=d_0\frac{\langle \Phi_{V,J}\rangle}{\langle \Phi_{J,V}\rangle},$$
where $d_0$ is any non zero complex number.

\begin{remark}\label{rmk:ambidex} From Lemma \ref{lemma:Hopf} and \ref{lemma:ambidex}, we see that $A=\lbrace V_{\alpha}, \alpha \in \overset{..}{\C} \rbrace$ and 
$$d(V_{\alpha})=d(\alpha)=\frac{\lbrace \alpha \rbrace}{\lbrace N\alpha \rbrace}$$
 is an ambidextrous pair.
 \end{remark}

The significance of ambidextrous pairs comes from the following theorem:

\begin{theorem}\label{thm:linkInv}\cite{GPT}
	Let $(A,d)$ be an ambidextrous pair. The following defines an invariant of framed oriented $\mathcal{C}$-colored links which have at least one component colored by an element in $A:$
	
	$$F(L)=d(RT_{\mathcal{C}}(L_e))$$
	where $e$ is an component of $T$ colored by $X\in A,$ $L_e$ is the tangle obtained from opening the link $L$ at $e,$ and $RT_{\mathcal{C}}$ is the Reshetikhin-Turaev functor associated to the ribbon category $\mathcal{C}.$
\end{theorem}

\section{Invariants of $3$-manifolds}
\label{sec:3mfd-inv}
To define an invariant of $3$-manifolds, we need extra assumptions on the ribbon category $\mathcal{C}.$ In Definition 4.2 of \cite{CGP14}, the authors introduce the concept of relative modular categories, which, as we will explain further down, allow the construction of an invariant of triples $(M,\omega,L)$ where $M$ is a $3$-manifold, $\omega$ is a cohomology class on $M$ and $L$ is a $\mathcal{C}$-colored framed oriented link, that satisfy some conditions.

We refer to \cite{CGP14} for the definition of relative modular categories. Our claim is that the category $\mathcal{C}$ satisfy the hypothesis of \cite[Definition 4.2]{CGP14}:  
\begin{theorem}
The ribbon category $\mathcal{C}$ is $\C/\Z$-modular relative to $\widetilde{X}=\frac{1}{4}\Z/\Z$ with modified trace $d$ and periodicity group $\Z$ realized by the modules $\sigma^k:=\varepsilon^{4k}$ with $k\in \Z.$
\end{theorem}
\begin{proof}
	We have to verify the axioms of Definition 4.2 of \cite{CGP14}. Axiom (1) simply says that $\mathcal{C}$ is $\C/\Z$-graded.
	We have seen that $\sigma\in \mathcal{C}_{\overline{0}}$ and since tensoring a module by $\sigma^k$ shifts all weights by $kN,$ the $\sigma^k,$ $k\in \Z$ give a free realization of $\Z$ in $\mathcal{C}_{\overline{0}}.$ 
	To check axiom (3), notice that since $E$ and $F$ vanish on $\sigma,$ one has for any $\overline{\alpha} \in \C/\Z$ and any $V\in \mathcal{C}_{\overline{\alpha}}$ that 
	$$c_{\sigma,V} \circ c_{V,\sigma}=q^{H\otimes H}=q^{N\alpha}Id_{V\otimes \sigma},$$
	since weight vectors in $V$ have weights congruent to $\alpha \mod \Z,$ and $q^{N\alpha}$ depends only on $\alpha \mod \Z.$ Finally, $\sigma$ has quantum dimension $1.$
	Axiom (4) is clear since $\widetilde{X}=\frac{1}{4}\Z/\Z$ is symmetrical and $\C/\Z$ is not covered by a finite number of its translates.
	Axiom (5) has been verified in Remark \ref{rmk:ambidex}.
	Axiom (6) is a consequence of Theorem \ref{thm:classifModules}: since $V_{\alpha}\otimes \sigma=V_{\alpha+N},$ for $\overline{\alpha}\in \C/\Z \setminus \widetilde{X},$ the simple objects of $\mathcal{C}_{\overline{\alpha}}$ are the $\Z$-orbits of $V_{\alpha},V_{\alpha+1}, \ldots, V_{\alpha+N-1}.$
	Axiom (7) follows from a similar computation to that of \cite[End of Section 2.2]{CGP14}, which gives the formula
	$\Delta_-=q^{3/2}G_N,$ where
	$$G_N=\underset{k=0}{\overset{N-1}{\sum}} q^{-\frac{1}{2}k^2}=\begin{cases}
		\sqrt{N} \ \textrm{if} \ N=1 \ \mathrm{mod} \ 4
		\\ -i\sqrt{N} \ \textrm{if} \ N=3 \ \mathrm{mod} \ 4
		\end{cases}.$$
	Finally, axiom (8) follows from the second equation of Lemma \ref{lemma:Hopf}. 
\end{proof}

We will now explain how the machinery of \cite{CGP14} constructs a $3$-manifold invariant from our category $\mathcal{C}.$ For $M$ a $3$-manifold and $T$ a framed trivalent graph in $M,$ let us a call a cohomology class $\omega \in H^1(M\setminus L,\C/\Z)$ \textit{non-integral} if $\omega \notin H^1(M\setminus L,\frac{1}{4}\Z/\Z).$ We will call a triple $(M,\omega,L)$ where $M$ is an oriented compact closed $3$-manifold, $\omega \in H^1(M\setminus L,\C/\Z)$ and $L$ is a framed oriented $\mathcal{C}$-colored link, a \textit{compatible triple}, if $\omega$ is non-integral and $c_e \in \mathcal{C}_{\omega(\mu_e)}$ for any component $L_e$ of $L,$ where $\mu_e$ is an oriented meridian of the component $L_e$ and $c_e$ is the color of $L_e.$

A surgery presentation $L_M$ of $M$ is then computable if for any component $L_{M,i}$ of $L_M,$ one has that $\omega(\mu_i)\notin \frac{1}{4}\Z/\Z,$ where $\mu_i$ is a meridian of $L_{M,i}.$ The results of \cite{CGP14} show that a computable surgery presentation exists for any compatible triple.

It follows from the general construction in \cite{CGP14} that we can define an invariant of $\mathcal{C}$-links as follows:

\begin{theorem}\label{thm:3-mfdInv}\cite[Theorem 4.7]{CGP14}
	Let $(M,\omega,L)$ be a compatible triple.  Let $L_M=L_{M,1}\cup \ldots \cup L_{M,n}$ be a computable surgery presentation for $(M,L,\omega),$ and set
	$$Z_N(M,\omega,L)=\frac{F_N(\underset{i=1}{\overset{n}{\bigcup}}(L_{M,i},\Omega_{\omega(m_i)})\cup L)}{\Delta_+^p\Delta_-^q},$$
	where $(p,q)$ is the signature of the linking matrix of $L_M,$ and for $\alpha \in \C\setminus \frac{1}{4}\Z,$ the Kirby color $\Omega_{\alpha}$ is
	$$\Omega_{\alpha}=\underset{i=0}{\overset{N-1}{\sum}}d(\alpha+i)V_{\alpha+i}.$$

	Then $Z_N(M,L,\omega)$ does not depend on the computable surgery presentation $L_M$ or on the choice of lifts of $\omega(m_i)$ to $\C,$ and thus defines an invariant of the triple $(M,\omega,L).$
\end{theorem}

\begin{remark}\label{rmk:compatible}
	Our definition of compatible triples is more restrictive than that of \cite{CGP14}, as they allow some triples with integral cohomology class, for which they can still define the invariant. We will consider only $3$-manifolds equiped with non-integral cohomology classes in the remainder of this note, hence we will stick to our more restrictive condition.
\end{remark}

\section{Maps on the skein module $S_{\zeta}(M)$}
\label{sec:maps}
We recall that we set $\zeta=-q^{1/2}=-e^{\frac{2i\pi}{N}},$ which is a primitive $2N$-th root of unity.
\begin{proposition}\label{prop:Skeinmodulemap}
	Let $M$ be a compact oriented $3$-manifold. Let $\omega\in H^1(M,\C/\Z)$ non integral. For $L$ a colored framed link with all components colored by $S_1,$ we set
	$$f_{\omega}(L)= (-1)^{\#L} Z_N(M,\omega,L)$$
	where $\#L$ is the number of components of $L.$ Then the map $f_{\omega}$  induces a map $S_{\zeta}(M)\longrightarrow \C.$
	
	In the above, since $S_1\in \mathcal{C}_{\overline{0}},$ the triple $(M,L,\omega)$ is compatible for any link $L,$ where the cohomology class $\omega\in H^1(M,\C/\Z)$ is intepreted as a cohomology class in $H^1(M\setminus L,\C/\Z)$ by restriction.
\end{proposition}
\begin{proof}

We need to see that $f_{\omega}$ satisfies the Kauffman relation at the $2N$-th root of unity $\zeta.$
Note that $S_1$ has quantum dimension $q+q^{-1},$ so $Z_N(M,\omega,L)$ is multiplied by $q+q^{-1}$ when adding a trivial component to $L.$ 
Let us now compute the braiding on $S_1\otimes S_1.$ Since $E^2$ and $F^2$ act as zero on $S_1,$ Formula \ref{eq:Rmatrix} reduces to 
$$R=q^{H\otimes H/2}(Id +(q-q^{-1})E\otimes F),$$
from which we get
$$c_{S_1,S_1}( e_0\otimes e_0)=q^{1/2}e_0\otimes e_0, \ c_{S_1,S_1}(e_1\otimes e_1)=q^{1/2}e_1\otimes e_1, $$
$$ c_{S_1,S_1}(e_0\otimes e_1)=q^{-1/2}e_1\otimes e_0, \ c_{S_1,S_1}( e_1\otimes e_0)= q^{-1/2}e_0\otimes e_1 + (q^{1/2}-q^{-3/2})e_1\otimes e_0.$$
and therefore
$$c_{S_1,S_1}^{-1}( e_0\otimes e_0)=q^{-1/2}e_0\otimes e_0, \ c_{S_1,S_1}^{-1}(e_1\otimes e_1)=q^{-1/2}e_1\otimes e_1, $$
$$ c_{S_1,S_1}^{-1}(e_0\otimes e_1)=(q^{-1/2}-q^{3/2})e_0\otimes e_1 + q^{1/2} e_1\otimes e_0, \ c_{S_1,S_1}^{-1}( e_1\otimes e_0)=q^{1/2}e_0 \otimes e_1 .$$
Therefore we get that $$q^{1/2}c_{S_1,S_1}-q^{-1/2}c_{S_1,S_1}^{-1}=(q-q^{-1})\mathrm{Id}_{S_1\otimes S_1}.$$
Recall that $\zeta=-q^{1/2}.$ We conclude that $f_{\omega}$ satisfies the skein relations:
\begin{center}
	\def\svgwidth{7cm}
\begingroup%
  \makeatletter%
  \providecommand\color[2][]{%
    \errmessage{(Inkscape) Color is used for the text in Inkscape, but the package 'color.sty' is not loaded}%
    \renewcommand\color[2][]{}%
  }%
  \providecommand\transparent[1]{%
    \errmessage{(Inkscape) Transparency is used (non-zero) for the text in Inkscape, but the package 'transparent.sty' is not loaded}%
    \renewcommand\transparent[1]{}%
  }%
  \providecommand\rotatebox[2]{#2}%
  \newcommand*\fsize{\dimexpr\f@size pt\relax}%
  \newcommand*\lineheight[1]{\fontsize{\fsize}{#1\fsize}\selectfont}%
  \ifx\svgwidth\undefined%
    \setlength{\unitlength}{134.58113543bp}%
    \ifx\svgscale\undefined%
      \relax%
    \else%
      \setlength{\unitlength}{\unitlength * \real{\svgscale}}%
    \fi%
  \else%
    \setlength{\unitlength}{\svgwidth}%
  \fi%
  \global\let\svgwidth\undefined%
  \global\let\svgscale\undefined%
  \makeatother%
  \begin{picture}(1,0.4371253)%
    \lineheight{1}%
    \setlength\tabcolsep{0pt}%
    \put(0,0){\includegraphics[width=\unitlength,page=1]{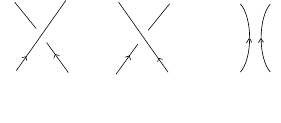}}%
    \put(-0.00038469,0.29623922){\color[rgb]{0,0,0}\makebox(0,0)[lt]{\lineheight{1.25}\smash{\begin{tabular}[t]{l}$\zeta$\end{tabular}}}}%
    \put(0.28730029,0.29088692){\color[rgb]{0,0,0}\makebox(0,0)[lt]{\lineheight{1.25}\smash{\begin{tabular}[t]{l}$-\zeta^{-1}$\end{tabular}}}}%
    \put(0.56083099,0.28999188){\color[rgb]{0,0,0}\makebox(0,0)[lt]{\lineheight{1.25}\smash{\begin{tabular}[t]{l}$=(\zeta^2-\zeta^{-2})$\end{tabular}}}}%
    \put(0,0){\includegraphics[width=\unitlength,page=2]{Kauffman.pdf}}%
    \put(0.22816229,0.05623876){\color[rgb]{0,0,0}\makebox(0,0)[lt]{\lineheight{1.25}\smash{\begin{tabular}[t]{l}$=-(\zeta^2+\zeta^{-2})$\end{tabular}}}}%
  \end{picture}%
\endgroup%

\end{center}
Note however that since $S_1$ is self dual, $f_{\omega}(L)$ does not depend on the orientations on the components of $L.$ If the two strands on the left hand side of the first skein relation belong to different components, then we get also:
\begin{center}
	\def\svgwidth{7cm}
\begingroup%
  \makeatletter%
  \providecommand\color[2][]{%
    \errmessage{(Inkscape) Color is used for the text in Inkscape, but the package 'color.sty' is not loaded}%
    \renewcommand\color[2][]{}%
  }%
  \providecommand\transparent[1]{%
    \errmessage{(Inkscape) Transparency is used (non-zero) for the text in Inkscape, but the package 'transparent.sty' is not loaded}%
    \renewcommand\transparent[1]{}%
  }%
  \providecommand\rotatebox[2]{#2}%
  \newcommand*\fsize{\dimexpr\f@size pt\relax}%
  \newcommand*\lineheight[1]{\fontsize{\fsize}{#1\fsize}\selectfont}%
  \ifx\svgwidth\undefined%
    \setlength{\unitlength}{130.62991766bp}%
    \ifx\svgscale\undefined%
      \relax%
    \else%
      \setlength{\unitlength}{\unitlength * \real{\svgscale}}%
    \fi%
  \else%
    \setlength{\unitlength}{\svgwidth}%
  \fi%
  \global\let\svgwidth\undefined%
  \global\let\svgscale\undefined%
  \makeatother%
  \begin{picture}(1,0.27291277)%
    \lineheight{1}%
    \setlength\tabcolsep{0pt}%
    \put(0,0){\includegraphics[width=\unitlength,page=1]{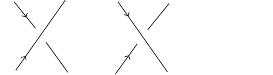}}%
    \put(-0.00254176,0.12883797){\color[rgb]{0,0,0}\makebox(0,0)[lt]{\lineheight{1.25}\smash{\begin{tabular}[t]{l}$\zeta^{-1}$\end{tabular}}}}%
    \put(0.29384495,0.12332378){\color[rgb]{0,0,0}\makebox(0,0)[lt]{\lineheight{1.25}\smash{\begin{tabular}[t]{l}$-\zeta$\end{tabular}}}}%
    \put(0.57564924,0.12240163){\color[rgb]{0,0,0}\makebox(0,0)[lt]{\lineheight{1.25}\smash{\begin{tabular}[t]{l}$=(\zeta^{-2}-\zeta^2)$\end{tabular}}}}%
    \put(0,0){\includegraphics[width=\unitlength,page=2]{Kauffman2.pdf}}%
  \end{picture}%
\endgroup%

\end{center}
From this equation and the previous one, we can deduce that:
\begin{center}
	\def\svgwidth{5cm}
\begingroup%
  \makeatletter%
  \providecommand\color[2][]{%
    \errmessage{(Inkscape) Color is used for the text in Inkscape, but the package 'color.sty' is not loaded}%
    \renewcommand\color[2][]{}%
  }%
  \providecommand\transparent[1]{%
    \errmessage{(Inkscape) Transparency is used (non-zero) for the text in Inkscape, but the package 'transparent.sty' is not loaded}%
    \renewcommand\transparent[1]{}%
  }%
  \providecommand\rotatebox[2]{#2}%
  \newcommand*\fsize{\dimexpr\f@size pt\relax}%
  \newcommand*\lineheight[1]{\fontsize{\fsize}{#1\fsize}\selectfont}%
  \ifx\svgwidth\undefined%
    \setlength{\unitlength}{117.0941883bp}%
    \ifx\svgscale\undefined%
      \relax%
    \else%
      \setlength{\unitlength}{\unitlength * \real{\svgscale}}%
    \fi%
  \else%
    \setlength{\unitlength}{\svgwidth}%
  \fi%
  \global\let\svgwidth\undefined%
  \global\let\svgscale\undefined%
  \makeatother%
  \begin{picture}(1,0.29802959)%
    \lineheight{1}%
    \setlength\tabcolsep{0pt}%
    \put(0,0){\includegraphics[width=\unitlength,page=1]{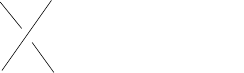}}%
    \put(0.2616999,0.12933296){\color[rgb]{0,0,0}\makebox(0,0)[lt]{\lineheight{1.25}\smash{\begin{tabular}[t]{l}$= \zeta$\end{tabular}}}}%
    \put(0,0){\includegraphics[width=\unitlength,page=2]{Kaufman3.pdf}}%
    \put(0.60872845,0.12034001){\color[rgb]{0,0,0}\makebox(0,0)[lt]{\lineheight{1.25}\smash{\begin{tabular}[t]{l}$+\zeta^{-1}$\end{tabular}}}}%
    \put(0,0){\includegraphics[width=\unitlength,page=3]{Kaufman3.pdf}}%
  \end{picture}%
\endgroup%

\end{center}
So the first Kauffman bracket skein relation holds if the two strands on the left hand side belong to different components. If they belong to the same component, then the first Kauffman relation also holds, as we have:
\begin{center}
	\def\svgwidth{9cm}
\begingroup%
  \makeatletter%
  \providecommand\color[2][]{%
    \errmessage{(Inkscape) Color is used for the text in Inkscape, but the package 'color.sty' is not loaded}%
    \renewcommand\color[2][]{}%
  }%
  \providecommand\transparent[1]{%
    \errmessage{(Inkscape) Transparency is used (non-zero) for the text in Inkscape, but the package 'transparent.sty' is not loaded}%
    \renewcommand\transparent[1]{}%
  }%
  \providecommand\rotatebox[2]{#2}%
  \newcommand*\fsize{\dimexpr\f@size pt\relax}%
  \newcommand*\lineheight[1]{\fontsize{\fsize}{#1\fsize}\selectfont}%
  \ifx\svgwidth\undefined%
    \setlength{\unitlength}{173.14160084bp}%
    \ifx\svgscale\undefined%
      \relax%
    \else%
      \setlength{\unitlength}{\unitlength * \real{\svgscale}}%
    \fi%
  \else%
    \setlength{\unitlength}{\svgwidth}%
  \fi%
  \global\let\svgwidth\undefined%
  \global\let\svgscale\undefined%
  \makeatother%
  \begin{picture}(1,0.19258202)%
    \lineheight{1}%
    \setlength\tabcolsep{0pt}%
    \put(0,0){\includegraphics[width=\unitlength,page=1]{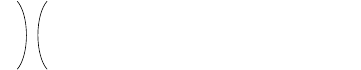}}%
    \put(-0.00191768,0.0922552){\color[rgb]{0,0,0}\makebox(0,0)[lt]{\lineheight{1.25}\smash{\begin{tabular}[t]{l}$\zeta$\end{tabular}}}}%
    \put(0.13428505,0.09225526){\color[rgb]{0,0,0}\makebox(0,0)[lt]{\lineheight{1.25}\smash{\begin{tabular}[t]{l}$=\zeta$\end{tabular}}}}%
    \put(0,0){\includegraphics[width=\unitlength,page=2]{Kauffman4.pdf}}%
    \put(0.31855937,0.09168294){\color[rgb]{0,0,0}\makebox(0,0)[lt]{\lineheight{1.25}\smash{\begin{tabular}[t]{l}$=\zeta^2$\end{tabular}}}}%
    \put(0,0){\includegraphics[width=\unitlength,page=3]{Kauffman4.pdf}}%
    \put(0.520002,0.09397211){\color[rgb]{0,0,0}\makebox(0,0)[lt]{\lineheight{1.25}\smash{\begin{tabular}[t]{l}$+$\end{tabular}}}}%
    \put(0,0){\includegraphics[width=\unitlength,page=4]{Kauffman4.pdf}}%
    \put(0.62644615,0.08538786){\color[rgb]{0,0,0}\makebox(0,0)[lt]{\lineheight{1.25}\smash{\begin{tabular}[t]{l}$=-\zeta^{-1}$\end{tabular}}}}%
    \put(0.85707516,0.09855033){\color[rgb]{0,0,0}\makebox(0,0)[lt]{\lineheight{1.25}\smash{\begin{tabular}[t]{l}$+$\end{tabular}}}}%
  \end{picture}%
\endgroup%

\end{center}
where we get the last equality as $\theta_{S_1}^{-1}=q^{-3/2}Id_{S_1}=-\zeta^{-3}Id_{S_1}.$
\end{proof}

\begin{proposition}\label{prop:classicalShadow}
	The map $f_{\omega}$ has classical shadow the abelian representation 
	$$\rho_{\omega}: \gamma\in \pi_1(M) \longrightarrow \begin{pmatrix}
	e^{4i\pi \omega(\gamma)} & 0 \\ 0 & e^{-4i\pi \omega(\gamma)}
	\end{pmatrix}.$$
\end{proposition}
\begin{proof}
We need to show that for $L$ a framed link in $M$ and $K$ a framed knot we have
$$f_{\omega}(L \cup T_N(K))=-\mathrm{Tr} \rho_{\omega}(K)f_{\omega}(L).$$
Notice that since $f_{\omega}$ induces a map $S_{\zeta}(M) \longrightarrow \C,$ it follows from \cite[Theorem 2]{Le15} that $f_{\omega}(L\cup T_N(K))$ depends only on the isotopy class of $K$ in $M.$ Therefore, we can assume that $K$ is a trivial knot in $S^3,$ possibly linked with the surgery link $L_M$ for $M$ and the link $L.$ Furthermore, if $K$ is not linked with $L_M$ then we can assume $K$ is a trivial knot in $S^3\setminus (L_M \cup K).$ In the latter case, the proposition follows as the quantum dimension of $T_N(S_1)$ is $T_N(q+q^{-1})=q^N+q^{-N}=2.$

Otherwise, the disk bounded by $K$ will link some strands of $L_M$ and $L,$ and at least one strand colored by a projective color $V_{\alpha}.$
However, projective objects in $\mathcal{C}$ form an ideal for the tensor product operation by \cite[Lemma 17]{GPV}, since this category is pivotal. Moreover, the total sum of colors of strand circled by $K$ is exactly $\omega(K)$ mod $\Z,$ since $S_1\in \mathcal{C}_{\overline{0}}$ and the color on a component $L_{M,i}$ of the surgery link has color $\Omega_{\omega(\mu_i)} \in \mathcal{C}_{\omega(\mu_i)},$ where $\mu_i$ is the meridian of $L_{M,i}.$ 

Thus we only need to consider the case where $K$ circles one strand colored by a projective object in $\mathcal{C}_{\omega(K)}.$ However, from Lemma \ref{lemma:HopfChebychev} and the fact the $N$-th Chebychev is an odd polynomial, we get that $$f_{\omega}(L\cup T_N(K))=-(e^{4i\pi\omega(K)}+e^{-4i\pi\omega(K)})f_{\omega}(L),$$
which concludes the proof.	
\end{proof}

The remainder of this section is devoted to the proof of the following proposition:
\begin{proposition}\label{prop:surjective}
	The map $f_{\omega}:S_{\zeta}(M)\longrightarrow \C$ is surjective for any non integral cohomology class $\omega.$
\end{proposition}
We will need a few preparatory lemmas.
\begin{lemma}\label{lemma:colorSelect}
	For any $\alpha \in \C\setminus \frac{1}{4}\Z,$ there exists a polynomial $Q_{\alpha}(z)\in \C[z]$ such that
	$$\Phi_{Q_{\alpha}(S_1),\Omega_{\alpha}}=Id_{V_{\alpha}}.$$
\end{lemma}
\begin{proof}
	Since $\Omega_{\alpha}=\underset{i=0}{\overset{N-1}{\sum}}d(\alpha+i)V_{\alpha+i},$ by Lemma \ref{lemma:Hopf}, this reduces to finding a polynomial $Q_{\alpha}(z)$ in $\C[z]$ such that $Q_{\alpha}(q^{\alpha}+q^{\alpha})=\frac{1}{d(\alpha)}$ and $$\forall i\in \lbrace 1,\ldots,N-1\rbrace, \ Q_{\alpha}(q^{\alpha+i}+q^{-\alpha-i})=0.$$ Note that the complex numbers $q^{\alpha+i}+q^{-\alpha-i}=2\cos(4\pi(\alpha+i)/N)$ are all distinct for $i=0,\ldots,N-1$ when $\alpha\notin \frac{1}{4}\Z.$ 
	
	Hence there exists such a polynomial by standard Lagrange interpolation.
\end{proof}
\begin{lemma}\label{lemma:crossingChange}
	For any $\alpha,\beta \in \C \setminus \frac{1}{4}\Z$ there exists polynomials $R_{\alpha,\beta,+},R_{\alpha,\beta,-}\in \C[z]$ such that  
	$$\Phi_{R_{\alpha,\beta\pm}(S_1),V_{\alpha}\otimes V_{\beta}}=\left(c_{V_{\alpha},V_{\beta}}\right)^{\pm 2}.$$
\end{lemma}
\begin{proof}
	If $\alpha +\beta \notin \frac{1}{4}\Z,$ then by Lemma \ref{lemma:tensorprod2}, 
	$$V_{\alpha}\otimes V_{\beta}\simeq \underset{k\in H_N}{\bigoplus} V_{\alpha+\beta+k}.$$  Furthermore, all the modules in this decomposition are simple.
	
	It is therefore enough to show that there is a polynomial $R_{\alpha,\beta,+} \in \C[z]$ such that $\Phi_{R_{\alpha,\beta,+}(S_1),V_{\gamma}}=\theta_{V_{\gamma}}$ for all $\gamma \in \alpha+\beta +H_r.$
	Note that $\theta_{V_{\gamma}}$ is proportional to $Id_{V_{\gamma}}$ since $V_{\gamma}$ is simple and that $\Phi_{S_1,V_{\gamma}}=(q^{\gamma}+q^{-\gamma})Id_{V_{\gamma}}$ by Lemma \ref{lemma:Hopf}.  We also claim that the $(q^{\gamma}+q^{-\gamma})=2\cos(4\pi\gamma/N)$ are all distinct for $\gamma \in \alpha+\beta+H_N,$ which would imply that this case follows from Lemma \ref{lemma:Hopf} and Lagrange interpolation.
	
	Indeed for $k,l\in H_N,$
	$$\cos(4\pi(\alpha+\beta+k)/N)=\cos(4\pi(\alpha+\beta+l)/N) \Leftrightarrow k=l \ \mathrm{mod} \ \frac{N}{2} \ \textrm{or} \ (\alpha+\beta+k)=-(\alpha+\beta+l) \ \mathrm{mod} \ \frac{N}{2}.$$
	Note that the second condition would imply that $\alpha+\beta \in \frac{1}{4}\Z,$ which we excluded, and the first condition implies that $k=l,$ since $N$ is odd.

	Now if $\alpha+\beta \in \frac{1}{4}\Z,$ let us write $\alpha+\beta=\frac{k}{4}$ for some $k\in \Z.$ 
	By Lemma \ref{lemma:tensorprod2}, 
	$$V_{\alpha}\otimes V_{\beta}\simeq \underset{n=0}{\overset{\frac{N-1}{2}}{\bigoplus}} P_{2n}\otimes \varepsilon^k,$$
	where we use the convention $P_{N-1}=V_0.$
	Note that $c_{\varepsilon^k,S_1} \circ c_{S_1,\varepsilon^k}=(-1)^k Id_{S_1\otimes \varepsilon^k}.$
	Hence, it suffices to find $R_{\alpha,\beta,+}\in \C[z]$ such that for any $n\in \lbrace 0, \ldots, (N-1)/2 \rbrace,$
	$$\Phi_{R_{\alpha,\beta,+}(S_1),P_{2n} \otimes \varepsilon^k}=\theta_{P_{2n}\otimes \varepsilon^k}.$$
	Note that $\Phi_{S_1,P_n\otimes \varepsilon^k}=(-1)^k(q^{n+1}+q^{-n-1})Id_{P_n\otimes \varepsilon^k} +(-1)^k(q-q^{-1})^2 (h_i\otimes Id_{\varepsilon^k})$ by Lemma \ref{lemma:Hopf}. Assume that $\theta_{P_n\otimes \varepsilon^k}=\mu_n Id_{P_n\otimes \varepsilon^k} +\nu_n (h_n\otimes Id_{\varepsilon^k})$ for some constants $\mu_n,\nu_n \in \C,$ what we want is then to find a polynomial $R_{\alpha,\beta,+}\in \C[z]$ such that for any $n\in \lbrace 0,\ldots,(N-3)/2\rbrace:$
	$$R_{\alpha,\beta,+}((-1)^k(q^{2n+1}+q^{-2n-1}))=\mu_n \ \textrm{and} \ R'_{\alpha,\beta,+}((-1)^k(q^{2n+1}+q^{-2n-1}))=\frac{\nu_n}{(q-q^{-1})^2}.$$
	and 
	$$R_{\alpha,\beta,+}((-1)^k(q^N+q^{-N}))=\mu_{N-1}.$$
	This is possible by Hermite interpolation, since the $(-1)^k(q^{2n+1}+q^{-2n-1})$ are all distinct for $n\in \lbrace 0, \ldots, (N-1)/2 \rbrace.$
\end{proof}

\begin{proof}[Proof of Proposition \ref{prop:surjective}]
	Let $L_M$ be a computable surgery presentation for $(M,\omega),$ and consider a diagram of $L_M.$ Up to isotopy, we can assume that there are crossings between each pair of components of $L_M.$ 
	
		We need to show that there is a framed link $L$ in $M$ such that $F_N(L_M\cup L)\neq 0,$ where $L_M$ is colored by Kirby colors as in Theorem \ref{thm:3-mfdInv} and $L$ is colored by $S_1.$ Equivalently, one can search for such a link colored by polynomials in $S_1.$

	We start by adding to $L_M$ a link $L'$ consisting of one meridian of each component of $L_M,$ where the meridian of the component $L_i$ is colored by $Q_{\alpha}$ if $L_i$ is colored by $\Omega_{\alpha}.$
	
	By Lemma \ref{lemma:colorSelect}, this reduces to showing that there is a link $L''$ such that $F_N(L_M' \cup L'') \neq 0,$ such that $L_M'$ is the link $L_M$ except components are now colored by modules $V_{\alpha}$ instead of the Kirby colors $\Omega_{\alpha},$ and $L''$ is colored by polynomials in $S_1.$ 
	
	Note that $L_M$ can be turned into a chain Hopf link (with some framing on components) by a sequence of crossing changes. By Lemma \ref{lemma:crossingChange}, we can realized those crossing changes by adding trivial circles around the crossings colored by some $R_{\alpha,\beta,\pm }(S_1).$ We take $L''$ to be the corresponding colored link.
	
	Now we have that $F_N(L_M \cup L' \cup L'')$ is equal to the invariant of a chain Hopf link (possibly with non-zero framing) colored by modules $V_{\alpha}$ with $\alpha \in \C\setminus \frac{1}{4}\Z.$ However, the invariant of such an Hopf link is non zero by the second equation of Lemma \ref{lemma:Hopf}. 

\end{proof}	
	
\begin{proof}[Proof of Theorem \ref{thm:main}]
	Theorem \ref{thm:main} is now a direct consequence of Proposition \ref{prop:classicalShadow} and \ref{prop:surjective}.
\end{proof}

\bibliographystyle{hamsalpha}
\bibliography{biblio}
\end{document}